\documentclass[11pt]{amsart} 
\usepackage{hyperref}
\usepackage{color}
\usepackage{amsmath, amsthm,url,
  amssymb,setspace,epsfig, mathrsfs,fontenc, 
  fullpage, MnSymbol}
\usepackage[alphabetic]{amsrefs}

\DefineSimpleKey{bib}{myurl}
\newcommand\myurl[1]{\url{#1}}
\BibSpec{webpage}{
  +{}{\PrintAuthors} {author}
  +{,}{ \textit} {title}
  +{}{ \parenthesize} {date}
  +{,}{ \myurl} {myurl}
}

\newtheorem{lem}{Lemma}
\newtheorem{exam}{Example}
\newtheorem{conj}{Conjecture}
\newtheorem{prop}{Proposition}
\newtheorem{thm}{Theorem}
\newtheorem{rem}{Remark}

\newtheorem{cor}{Corollary}
\newtheorem{defe}{Definition}

\newcommand{\nc}{\newcommand}
\nc{\renc}{\renewcommand}
\nc{\ssec}{\subsection}
\nc{\sssec}{\subsubsection}
\nc{\on}{\operatorname}
\nc {\ra} {\rightarrow}
\nc {\inj} {\hookrightarrow}
\nc {\lra} {\longleftrightarrow} 
\nc         {\rar}[1]       {\stackrel{#1}{\longrightarrow}}

\nc {\Sl}{\mathfrak{sl}}
\nc{\fm}{\mathfrak{m}}
\nc{\fg}{\mathfrak{g}}
\nc{\ft}{\mathfrak{t}}
\nc{\fk}{\mathfrak{k}}
\nc{\fI}{\mathfrak{i}}
\nc{\fh}{\mathfrak{h}}
\nc{\fb}{\mathfrak{b}}
\nc{\fu}{\mathfrak{u}}
\nc{\fn}{\mathfrak{n}}
\nc{\hfg}{\widehat{\fg}}
\nc{\hfh}{\widehat{\fh}}

\nc {\hH}{{\check{H}}}
\nc {\hB}{{\check{B}}}
\nc {\hN}{{\check{N}}}
\nc {\hG}{{{\check{G}}}}
\nc {\cfg}{\check{\fg}}
\nc {\cfb}{\check{\fb}}
\nc {\cfn}{\check{\fn}}
\nc {\cfh}{\check{\fh}}
\nc {\clambda}{\check{\lambda}}

\nc {\bone}{\mathbf{1}}
\nc {\bu}{\mathbf{u}}
\nc {\bv}{\mathbf{v}}
\nc {\bw}{\mathbf{w}}

\nc{\bC}{\mathbb{C}}
\nc{\bM} {\mathbb{M}}
\nc{\bU} {\mathbb{U}}
\nc{\bV} {\mathbb{V}}
\nc{\bW}{\mathbb{W}}
\nc{\bL}{\mathbb{L}}
\nc {\bZ} {\mathbb{Z}}
\nc {\bGm} {\mathbb{G}_m}

\nc {\cD}{\mathcal{D}}
\nc {\cE}{\mathcal{E}}
\nc {\cF}{\mathcal{F}}
\nc {\cG}{\mathcal{G}}
\nc {\cK}{\mathcal{K}}
\nc{\cZ}{\mathcal{Z}}
\nc {\cDt}{\cD^\times} 
\nc {\cA} {\mathcal{A}}
\nc {\cI}{\mathcal{I}}
\nc {\cR}{\mathcal{R}}
\nc {\cJ}{\mathcal{J}}
\nc {\cS}{\mathcal{S}}
\nc {\cO}{\mathcal{O}}
\nc {\cP}{\mathcal{P}}

\nc {\tU} {\widetilde{U}}

\nc {\Ind}{\mathrm{Ind}}
\nc {\RS}{\mathrm{RS}}
\nc {\Fun}{\mathrm{Fun}}
\nc {\Lie}{\mathrm{Lie}}
\nc {\Hom}{\mathrm{Hom}}
\nc {\Vac}{\mathrm{Vac}}
\nc {\crit}{\mathrm{crit}}
\nc {\Res}{\mathrm{Res}} 
\nc {\ev}{\mathrm{ev}}
\nc{\Spec}{\mathrm{Spec}\,}
\nc {\ord}{\mathrm{ord}}
\nc {\Op}{\mathrm{Op}}
\nc{\Loc}{\mathrm{Loc}}
\nc {\MOp}{\mathrm{MOp}}
\nc {\Conn}{\mathrm{Conn}}
\nc {\MT}{\mathrm{MT}}
\nc{\Sym}{\mathrm{Sym}}
\nc {\modd}{\mathrm{mod}}
\nc {\ad}{\mathrm{ad}}
\nc {\Tr}{\mathrm{Tr}}
\nc{\Kil}{\mathrm{Kil}}
\nc {\End}{\mathrm{End}}
\nc {\can}{\mathrm{can}}
\nc {\Gr}{\mathrm{Gr}}
\nc {\Aut}{\mathrm{Aut}}
\nc {\Der}{\mathrm{Der}}
\nc {\gr}{\mathrm{gr}}
\nc {\sm}{\mathrm{sm}}
\nc {\Map}{\mathrm{Map}}
\nc {\Hitch}{\mathrm{Hitch}}
\nc {\Higgs}{\mathrm{Higgs}}
\nc {\T}{\mathrm{T}}
\nc{\Bun}{\mathrm{Bun}}

\nc {\bP} {\bar{P}}
\nc {\bJ}{\bar{J}}

\usepackage[all]{xy}

\nc {\quo}{\mathopen{ /\!/}}
\nc {\llp} {\mathopen{ (\!(}}
\nc {\rrp} {\mathopen{ )\!)}}
\nc {\llb} {\mathopen{ [\![}}
\nc {\rrb} {\mathopen{ ]\!]}}
\nc {\lc} {\mathopen{:\!}}
\nc {\rc}{\mathopen{\!:}}
\nc{\fgbb} {\fg\llb t \rrb}
\nc{\fgpp} {\fg\llp t \rrp}
\nc{\fhbb} {\fh\llb t \rrb}
\nc{\fhpp} {\fh\llp t \rrp}
\nc{\bCpp} {\bC \llp t \rrp}
\nc{\bCbb} {\bC \llb t \rrb}
\nc{\hNpp} {\hN \llp t \rrp}
\nc{\hNbb} {\hN \llb t \rrb}
\nc{\cfbpp} {\cfb \llp t \rrp}
\nc{\cfbbb} {\cfb \llb t \rrb}

\nc {\Gpp} {G\llp t \rrp}
\nc {\Npp} {N\llp t \rrp}
\nc {\Gbb} {G\llb t \rrb}
\nc {\cC} {\mathcal{C}}
\nc {\cB}{\mathcal{B}}
\nc{\ocB}{\bar{\cB}}
\nc{\ocA}{\bar{\cA}}
\nc {\bR}{\mathbb{R}}
\nc {\bQ} {\mathbb{Q}}
\nc {\cL}{\mathcal{L}}
\nc {\bI}{\bar{I}}

\nc{\tg} {\mathtt{g}}
\nc {\tc}{\mathtt{c}}

\nc {\oG} {\overline{G}}
\nc{\ofg}{\overline{\fg}} 

\nc{\Fq} {\mathbb{F}_q}
\nc{\Fqt}{\Fq\llp t \rrp}
\nc{\ocK}{\overline{\cK}}
\nc{\Gal}{\mathrm{Gal}}
\nc{\uG}{\underline{G}}

\newcommand{\quash}[1]{}  %%Anything in \quash is ignored
\begin{document}
\title{Preservation of depth in local  geometric Langlands correspondence} 
\author{Tsao-Hsien Chen and Masoud Kamgarpour}

\email{chenth@math.northwestern.edu}
%\address{Department of Mathematics, Northwestern University} 
\email{masoud@uq.edu.au}
%\address{School of Mathematics and Physics, The University of Queensland}

\date{\today}

\subjclass[2010]{17B67, 17B69, 22E50, 20G25}

\keywords{Local geometric Langlands, Moy-Prasad Theory, slope of connections, opers, affine vertex algebras, Segal-Sugwara operators, Hitchin's fibration, twisted groups}

\begin{abstract}
It is expected that, under mild conditions,  local Langlands correspondence preserves depths of representations. In this article, we formulate a conjectural geometrisation of this expectation. We prove half of this conjecture by showing that the depth of a categorical representation of the loop group is less than or equal to the depth of its underlying geometric Langlands parameter. A key ingredient of our proof is a new definition of the slope of a meromorphic connection, a definition which uses opers. In the appendix, we consider a relationship between our conjecture and Zhu's conjecture on non-vanishing of the Hecke eigensheaves produced by Beilinson and Drinfeld's quantisation of Hitchin's fibration for non-constant groups.
\end{abstract} 

\maketitle 

\tableofcontents

\newpage

\section{Introduction}

Let $F$ be a local non-Archimedean field like $\mathbb{Q}_p$ or $\mathbb{F}_q\llp t \rrp$ and $G$  a connected reductive group over $F$. The local Langlands correspondence is a conjectural relationship between two types of data. The first data are, roughly speaking, equivalence classes of homomorphisms from the absolute Galois group of $F$ to $\hG$, the Langlands dual group of $G$. These are sometimes called the \emph{local Langalnds parameters}. The second data are the isomorphism classes of smooth irreducible representations of $G(F)$. 

Using the upper numbering ramification groups \cite{Serre}, one can define the notion of depth for local Langlands parameters. On the other hand, the notion of depth for smooth representations of $G(F)$ was defined by Moy and Prasad \cite{MP}, using the Bruhat-Tits Theory \cite{BT1, BT2}. It is expected that in most circumstances, the local Langlands correspondence preserves depth, cf. \cite{Yu-Ottawa}. The purpose of this paper is to examine the \emph{geometric} analogue of this expectation. To start, we give a leisurely introduction to Frenkel and Gaitsgory's proposal for geometrising the local Langlands correspondence. 

\ssec{Local geometric Langlands} Henceforth, let  $G$ be a simple complex group of adjoint type and  $\fg= \Lie(G)$. Let $\hG$ and $\cfg$ denote the Langlands dual objects. Let $\Gpp$ and $\fgpp$ denote the formal loop group and formal loop algebra. Let $\cK=\bCpp$ and let $\cDt=\Spec(\bCpp)$ denote the punctured formal disk.

 It has been known for a long time there is a deep and mysterious analogy between Galois representations and local systems, cf. appendix of \cite{Katz87}. By definition, a $\hG$-local system on $\cDt$ is a $\hG(\cK)$-gauge equivalence classes of operators of the form 
 \begin{equation} 
 \nabla = \partial_t + A, \quad \quad A \in \cfg(\cK)=\cfg\otimes \bCpp
 \end{equation} 
The gauge action of $g\in \hG(\cK)$ is defined by 
\begin{equation} 
g.(\partial_t + A) = \partial_t + gAg^{-1} - (dg) g^{-1}.
\end{equation} 
We let $\Loc_\hG(\cDt)$ denote the set of isomorphism classes of $\hG$-local systems on $\cDt$. It is widely accepted that these are the geometric analogues of the local Langlands parameters.

On the other side of the Langlands correspondence, Frenkel and Gaitsgory propose that the geometric analogue of smooth representations should be \emph{categorical representations} of  $G\llp t \rrp$.\footnote{There are other proposals for geometrising the local Langlands correspondence, cf.  \cite{Beilinson}, \cite{Langlands}.} We refer the reader to \cite[\S 20]{FG} and \cite[\S 1.3]{FrenkelBook} for this notion. 
A good toy model to keep in mind, for our purposes, is the action of $G$ on the category of $\fg$-modules. In more details, the group $G$ acts on its Lie algebra $\fg$ via the adjoint action. Therefore, every $g\in G$ acts on the category $\fg-\modd$   by sending a representation $\fg\ra \End(V)$ to the composition
\[
\fg\rar{\mathrm{Ad(g)}} \fg \ra \End(V).
\]
 This is an example of a categorical action. It is also possible to ``decompose" this categorical representation. Namely, let $\cZ(\fg)$ denote the centre of the universal enveloping algebra of $\fg$. For every character $\chi$ of $\cZ(\fg)$, let $\fg-\modd_\chi $ denote the full subcategory of $\fg-\modd$ consisting of those modules on which $\cZ(\fg)$ acts by the character $\chi$. Then $\fg-\modd_\chi$ is preserved under the action of $G$; thus, it is a sub-representation of $\fg-\modd$. 

We are interested, however, in categorical representations of the \emph{loop} group; thus, we should look at the action of $\Gpp$ on the category of $\fgpp=\fg\otimes \bCpp$-modules. Actually, it is fruitful to consider not the loop algebra itself but its universal central extension known as the \emph{affine Kac-Moody algebra} $\hfg$. Recall that representations of $\hfg$ have a parameter, an invariant bilinear form $\kappa$ on $\fg$, which is called the \emph{level}. We let $\hfg_\kappa-\modd$ denote the category (smooth) representations of $\hfg$ at the critical level. By a similar reasoning as in the previous paragraph, $\hfg_\kappa-\modd$ carries a natural action of the loop group $\Gpp$.

Representations of $\hfg$ corresponding to the bilinear form $\kappa=c$ which is equal to minus one half of the Killing form are called representations at the \emph{critical level}. The advantage of the critical level is that here the  (completed) universal enveloping algebra of $\hfg$ acquires a large centre. Thus, we may ``decompose" the representation $\hfg_c-\modd$ using the characters of the centre. 

More precisely, according to a  remarkable theorem of Feigin and Frenkel \cite{FF92}, the centre $\cZ_c$ of the completed universal enveloping algebra of the affine Kac-Moody algebra $\hfg$ at the critical level identifies canonically with the algebra of functions on the ind-scheme $\Op_\hG(\cDt)$ of \emph{$\hG$-opers} over the formal punctured disk. Thus, every point $\chi\in \Op_\hG(\cDt)=\Spec(\cZ_c)$ defines a character of the centre, and therefore, a categorical representation $\hfg_c-\modd_\chi$ of the loop group. These are suppose to be the categorical analogues of smooth representations of $p$-adic groups. In particular, the Grothendieck group of $\hfg_c-\modd_\chi$ should ``look like" a smooth representation.\footnote{According to \cite{FG}, these categorical representations also have descriptions in terms of (twisted) D-modules on generalised flag varieties. We will not use this alternative description.}

Having defined the geometric analogue of Langlands parameters and smooth representations, let us now relate them to each other.  To this end, we recall that opers are local systems plus additional data (see \S \ref{ss:conj} and \S \ref{s:slope} for more information on opers). Hence, one has a canonical forgetful map 
\begin{equation}
p:\Op_\hG(\cDt)\ra \Loc_\hG(\cDt).
\end{equation}
 The main result of \cite{FZ} states that this map is surjective. 
 
 It follows that for every geometric Langlands parameter $\sigma \in \Loc_\hG(\cDt)$, one has, in principal, many categorical representations of $\Gpp$; namely, the representations $\hfg_c-\modd_\chi$ where $\chi\in p^{-1}(\sigma)$.  
Frenkel and Gaitsgory conjecture that these categorical representations are equivalent. 
 Thus, given $\sigma$, there exists a canonical category $\cC_\sigma$ equipped with the action of $\Gpp$; moreover, $\cC_\sigma$ is equivalent to $\hfg_c-\modd_\chi$ for every $\chi\in p^{-1}(\sigma)$.  This is Frenkel and Gaitsgory's conjectural geometrisation of the local Langlands correspondence \cite{FG}.\footnote{Frenkel and Gaitsgory make their conjecture precise by exploiting connections with the global geometric Langlands correspondence. We will not consider this global characterisation.}  
 
 In a series of papers \cite{FG, FG09, FG09b,FG09c}, Frenkel and Gaitsgory examined the unramified and tamely ramified parts of the local geometric Langlands correspondence. These cases correspond to $\sigma$ being trivial or regular singular with unipotent monodromy. Almost all the results obtained in \emph{op. cit} are about the Iwahori integrable part of the theory.\footnote{For a representation of $\hfg$, being Iwahori-integrable is closely related to being in category $\cO$.}  As far as we know, very little is known about the correspondence for general $\sigma$, or even for unramified or tamely ramified $\sigma$, but beyond the Iwahori integrable situation. 
 We hope that the point of view of this text will be useful for further investigations of these cases.

\ssec{Main conjecture}\label{ss:conj}  We now explain how to geometrise the expectation that the local Langlands correspondence preserves depth. It turns out that $\hG$-local systems have a numerical invariant, called \emph{slope}, which is a natural candidate for the geometric analogue of depth of Langlands parameters. This notion goes back to the work of Katz and Deligne in early 70s. We refer the reader to \S \ref{s:slope} for a thorough discussion of various definitions of slope and the history of this invariant. For now, we give a definition of slope which we learned from  \cite{FGross}. A $\hG$-local system $\sigma$ has slope $a/b$ if the following holds. Pass to the extension given by adjoining the $b^\mathrm{th}$ root of $t$: $u^b=t$. Then the local system, written using the parameter $u$ in the extension, should have in its gauge equivalence class a representative which has a pole of order $a+1$ and its top polar part should not be nilpotent. We denote the slope of $\sigma$ by $s(\sigma)$.

 On the other side of the Langlands correspondence, it is straightforward to generalize Moy and Prasad's definition of depth to the categorical setting. Let us first recall the classical definition. In \cite{BT1, BT2}, Bruthat and Tits associated to $G$ a combinatorial object known as the \emph{Bruhat-Tits building} $\cB(G)$. For every $x\in \cB(G)$ and $r\in \bR_{\geq 0}$, Moy and Prasad \cite{MP} defined a subgroup $G_{x,r^+}\subset G(F)$. In addition, they defined the depth of a smooth representation of $G(F)$ by 
 \[
 \mathrm{depth}(V):=\inf \{ r\in \bR_{\geq 0}\, | \, \exists \, \textrm{$x\in \cB(G)$ such that $V^{G_{x,r^+}}$ is non-trivial}\}, 
 \]
 where $V^{G_{x,r^+}}\subseteq V$ denote the subspace $V$ consisting of vectors fixed under $G_{x,r^+}$. 
 
  It is easy to categorify the above definition. First of all, thanks to \cite{Yu-Model}, one knows that $G_{x,r^+}$ come equipped with a canonical \emph{smooth model}. In particular, this means that one can realise $G_{x,r^+}$ as the group of $\bC$-points of a proalgebraic group over $\bC$. Now if $\cC$ is a categorical representation of the loop group, we defined the depth $\cC$ by
 \begin{equation} 
 d(\cC):=\inf \{ r\in \bR_{\geq 0}\, | \, \exists \, \textrm{$x\in \cB(G)$ such that $\cC^{G_{x,r^+}}$ is non-trivial}\}. 
 \end{equation} 
  Here $\cC^{G_{x,r^+}}$ denotes the ``category of $G_{x,r^+}$ strongly equivariant objects" of $\cC$, cf. \cite[\S 20]{FG}, \cite[\S 10]{FrenkelBook}.

 In view of the above discussions, the following conjecture is the geometric analogue of the expectation that the local Langlands correspondence preserves depth. 

\begin{conj} \label{c:main}
Let $\sigma\in \Loc_\hG(\cDt)$ and let $\chi\in \Op_\hG(\cDt)$ be an oper whose underlying local system is $\sigma$ (i.e., $p(\chi)=\sigma$). Then 
\[
s(\sigma)=d(\hfg_c-\modd_\chi).
\] 
\end{conj} 

%Thus, even though the analogies employed between characteristic zero and positive characteristic are ``meta-mathematical", guided by these analogies, we have arrived at a precise conjecture relating numbers coming from two, a priori, different fields of mathematics.

\ssec{Main results} In this paper, we prove one-half of the above conjecture. We will also make partial progress towards proving the other half. The key ingredient is a new definition of slope of local systems, a definition which uses opers. For now, the only fact we need to know about opers is that one can represent an oper $\chi\in \Op_\hG(\cDt)$ with an ordered $\ell$-tuple $(v_1,\cdots, v_\ell)$ where $v_i\in \bCpp$ and $\ell$ is the rank of $G$. Let us write $v_j=t^{-n_j}.h_i$ where $h_i\in \bCbb^{\times}$. Let $d_i$, $i=1,\cdots, \ell$, denote the exponents of the Lie algebra $\hfg$. 

 \begin{defe}\label{slope via oper}
  The slope of $\chi$ is defined by 
\[
s(\chi):= \sup \{ 0, \sup\{\frac{n_i}{d_i+1}-1\}_{i=1,\cdots, \ell} \}. 
\]
\end{defe} 

We let $\Op_\hG^r\subset \Op_\hG(\cDt)$ denote the subscheme of opers of slope less than or equal to $r$. Note that if $n$ is a positive integer, then $\Op_\hG^n$ equals the space $\Op_\hG^{\ord_n}$ of opers on $\cD$ with singularity less than or equal to $n$, cf. \cite[\S 3.7.7]{BD}, \cite{FG}. 
The following result states that the slope of an oper equals the slope of its underlying connection. 
\begin{prop} \label{p:slopeOper}
Let $\chi\in p^{-1}(\sigma)$. Then $s(\chi)=s(\sigma)$. 
\end{prop} 

The following corollary is an immediate consequence. It appears to be a new result in the theory of meromorphic connections.\footnote{We note, however, that this result would follow from the Bremer-Sage Theory together with an unpublished result of Yu; see Remark \ref{r:Yu}. For smooth representations of $p$-adic groups, the analogous result 
is proved in \cite{Reeder-Yu}.}

\begin{cor} For every $\sigma \in \Loc_\hG(\cDt)$, the slope $s(\sigma)$ divides a fundamental degree of $\cfg$. 
\end{cor} 

We refer the reader to Section \ref{s:slope} for a thorough discussion of the slope.

We are now ready to state our main result. 

\begin{thm} \label{t:main}
For all $\chi\in \Op_\hG(\cDt)$, we have $s(\chi) \leq d(\hfg_c-\modd_\chi)$.  
\end{thm}

\ssec{Idea of the proof}
Let us briefly explain the main ingredient in the proof of this theorem. First of all, one can show (cf. \S 10 \cite{FrenkelBook}) that in the present situation, the categorical depth of $\hfg_c-\modd_\chi$ can be alternatively defined by 
\[ 
d(\hfg_c-\modd_\chi)=
\inf \{ r\in \bR_{\geq 0}\, | \, \exists \, \textrm{$x\in \cB(G)$ such that $\hfg_c-\modd_\chi$ contains a $G_{x,r^+}$-integrable module}.\}
\]

Suppose $d(\hfg_c-\modd_\chi)\leq r$. Then $\hfg_c-\modd_\chi$ contains a $G_{x,r^+}$-integrable module $W$. To show that $s(\chi)\leq r$, it is enough to show that $W$ is centrally supported on the subscheme $\Op_\hG^{r}$. It follows from Kolchin's theorem (cf. Remark \ref{r:intDepth}) that we have a canonical non-zero morphism of $\bU_{x,r}\ra W$, where 
\begin{equation}
\bU_{x,r}:=\Ind_{\fg_{x,r^+}\oplus \bC}^{\hfg_c}(\bC). 
\end{equation} 
Thus, Theorem \ref{t:main} follows from the following: 

\begin{thm} \label{t:support} 
The natural morphism $\bC[\Op_\hG(\cDt)]\simeq \cZ_c \ra \End_{\hfg_c}(\bU_{x,r})$ factors through the quotient $\bC[\Op_\hG(\cDt)]\ra \bC[\Op_\hG^r]$. 
\end{thm}  

If $x$ is the hyperspecial vertex and $n$ is a non-negative integer, then $\fg_{x,n^+}=t^{n+1}\fgbb$. In this case, the above theorem is due to Beilinson and Drinfeld \cite[\S 3.8.7]{BD}. In view of the previous discussion, we can rephrase the theorem of Beilinson and Drinfeld as stating that 
\begin{equation} s(\chi)\leq \lceil d(\hfg_c-\modd_\chi) \rceil,
\end{equation}
 where $\lceil x\rceil$ denotes the smallest integer greater than or equal to $x$. Our main theorem, therefore, sharpens Beilinson and Drinfeld's theorem by removing $\lceil - \rceil$.

We prove Theorem \ref{t:support} by using basic properties of Segual-Sugawara vectors along with some general properties of Fourier coefficients of vertex fields.  We refer the reader to Section \ref{s:rep} for the details of the proof.

\ssec{Towards establishing the converse}\label{ss:global} 
Suppose $\chi$ is an oper with slope less than or equal to $r$.
How should one prove the inequality $d(\hfg_c-\modd_\chi)\leq r$? Suppose we can produce a module $\bV\in \hfg_c-\modd$ such that 
\begin{enumerate}
 \item[(i)] $\bV$ is $G_{x,r^+}$-integrable (thus, centrally supported on $\Op_\hG^r$); 
 \item[(ii)] The central reduction $\bV(\chi)$ is non-zero. 
 \end{enumerate} 
 Then $\bV(\chi)$ is a $G_{x,r^+}$-integrable object of $\hfg_c-\modd_\chi$, implying that $d(\hfg_c-\modd_\chi)\leq r$. 

As a motivating example, consider the module $\bV_n=\Ind_{t^{n}\fgbb\oplus \bC}^{\hfg_c}(\bC)$. According to Lemma 7.2.2 \cite{FG}, this module is \emph{free} over its central support $\Op_\hG^n=\Op_\hG^{\ord_n}$. Thus, for all $\chi\in \Op^{\ord_n}_\hG$, we have $\bV_n(\chi)\neq 0$ which implies that $d(\hfg_c-\modd_\chi)\leq n$. From this, it follows easily that
\begin{equation} 
 d(\hfg_c-\modd_\chi)\leq \lceil s(\chi) \rceil .
\end{equation} 
Our challenge (Conjecture \ref{c:main}) is to sharpen the above inequality by removing the $\lceil-\rceil$. 

The aforementioned freeness result was first pointed out by Drinfeld, who deduced it from the flatness of the Hitchin's map. Subsequently, Eisenbud and Frenkel gave a purely local proof using results of Mustata on singularities of jet schemes \cite[Appendix A]{EFM}. At the moment, we do not know how to extend this purely local approach to a more general setting. In the appendix, we sketch how one can potentially use the Hitchin's fibration for \emph{twisted groups} for constructing the desired module $\bV$.

%\ssec{Organization of the text} 
%In Section \ref{s:slope}, we give brief description of four different but equivalent definitions of the slope of local systems. Section \ref{s:rep} is mostly concerned with the Proof of Theorem \ref{t:support'}. As a side remark, we also define the notion of depth for smooth representations of affine Kac-Moody algebras, and relate this notion (at the critical) to the slope of central support of the modules. 

\ssec{Acknowledgement} 
We would like to thank C. Bremer, A. Molev, D. Sage, Z. Yun and X. Zhu for helpful conversations.
The first author learned the definition of slope via opers, which is crucial in this paper, from X. Zhu. He is happy to thank him.
The second author was supported by the Australian Research Council Discovery Early Career Research Award.

%%%%%%%%%%%%%%%%%%%%%%%%%%%%%%%%%%%%%%%%%%%%%%%%

%%%%%%%%%%%%%%%%%%%%%%%%%%%%%%%%%%%%%%%%%%%%%%%%%

\section{Slope of meromorphic connections} \label{s:slope}
\ssec{Overview} 
Let $G$ be a connected reductive group over the complex numbers. The notion of slope for $G$-local systems on $\cDt$ has a long and complicated history. In \S 11.9 of \cite{Katz70}, Katz explains what irregular connections on \emph{vector bundles} over $\cDt$ are, building on earlier works of Fuchs, Turrittin and Lutz. In particular, he explains  how to attach a canonical rational number to every irregular connection. For this reason, the slope is sometimes known as the \emph{Katz invariant}. 
The same concept also appears in  \cite{Deligne70}, Section II, \S 1. 

One of the characterisations of the slope of vector bundles (the one involving gauge transformation by elements in $G\llp t^{1/b} \rrp$) can be generalised, verbatim, to the case of connections on $G$-bundles.  It seems that this generalisation was first considered in \cite{Varadarajan}. We review the Katz-Deligne-Babbit-Varadajan definition of slope in \S \ref{ss:slopeDef} and give a short proof of the fact that it is well-defined by using opers. This proof, however, uses a nontrivial theorem of Frenkel and Zhu \cite{FZ}. 

As mentioned in the introduction, there is a deep analogy between Galois representation and flat connections. Guided by this analogy, Katz  \cite{Katz87} defined the \emph{differential Galois group}, by employing the Tannakian structure on the category of connections. 
It is clear from this formulation that the notion of slope of a flat vector bundle, defined in \emph{op. cit.} via filtration subgroups,  extends to flat $G$-bundles. We note, however, that the structure of the differential Galois group of $\cDt$ and its filtration subgroups are not easy to discern. We review the Tannakian definition of slope in \S \ref{ss:slopeTan}. 

In \cite{BS}, the authors define the slope of flat bundles using Moy-Prasad Theory. In more details, Bremer and Sage define what it means for a flat bundle $\sigma$ to contain a ``strata". They prove that the slope of $\sigma$ is the minimum of depths of a fundamental strata contained in it. In addition, they provide an algorithm for determining the slope and define a canonical form for connections over the base field (as oppose to going to a field extension of the form $\bC\llp t^{1/b} \rrp$.) 
Their approach makes clear the analogy between slopes of local systems and the depth of smooth representations (or categorical representations). We don't know, however, how to prove any relationship between depths of (categorical) representations and slopes of local systems using their definition.

We use the notion of oper to define the slope of flat $G$-bundles. One of the advantages of our definition is that it will be \emph{obvious} that the denominator of the slope of a $G$-bundle is a divisor of a fundamental degree of the Lie algebra of $G$.  Another advantage is that we can use this definition to make progress on Conjecture \ref{c:main}. The disadvantage is that one does not have an algorithm for putting a connection in its oper form. (The proof in \cite{FZ} is non-constructive.)

\ssec{First definition of slope} \label{ss:slopeDef} 
We start by recalling some basic definitions. Let $G$ be a connected reductive group over $\bC$. Let $\sigma \in \Loc_G(\cDt)$. By definition, 
$\sigma$ consists of a pair $(\cF,\nabla)$, where $\cF$ is a $G$-bundle on $\cD^\times$ and $\nabla$ is a meromorphic connection on $\cF$. One knows that every bundle on $\cDt$ is trivial. Choosing a trivialization for $\cF$, we can write the connection $\nabla$ as 
\begin{equation}\label{eq:op}
\nabla = \partial_t + A,\quad \quad A=A_{-n}t^{-n} + A_{-n+1}t^{-n+1} + \cdots, \quad \quad A_i\in \fg, \quad A_{-n}\neq 0. 
\end{equation} 
The integer $n$ and the element $A_{-n}$ are called the \emph{order of singularity} and the (top) \emph{polar part} of this trivialization, respectively.  
Changing the trivialization of $\cF$ by $g\in \Gpp$\footnote{Following a common abuse of notation, we are writing $\Gpp$ when we really mean $G(\bCpp)$.}  corresponds to a gauge transformation of the above expression 
\[
\nabla \mapsto g.\nabla :=\partial_t +gAg^{-1} - (\partial_t g)g^{-1}
\]
Thus, one can alternatively define $\sigma$ as a $\Gpp$-gauge equivalence class of operators of the form \eqref{eq:op}. 
After \cite{Katz70} and \cite{Deligne70}, one says that $\sigma=(\cF,\nabla)$ is \emph{regular} (resp. \emph{regular singular}) 
if in a particular trivialization of $\cF$, the order of singularity of the connection $\nabla$ is zero (resp. one). Otherwise, we say that $\sigma$ is \emph{irregular}.  

Note that the order of singularity of $\nabla$ is not invariant under gauge transformation: given an operator in the form \eqref{eq:op}, it may be possible to find a different trivialisation in which the order of singularity of $\sigma$ is less than $n$. The following lemma states that this cannot happen if $A_{-n}$ is \emph{nilpotent}.

\begin{lem}\label{order of singularity}
 Assume that the order of singularity of $\nabla$ is
$n\geq 2$ and the polar part of $\nabla$ is non-nilpotent. 
Then  every operator in the gauge equivalence class of $\nabla$ will have oder of singularity greater than or
equal to $n$. 
\end{lem}
\begin{proof}
Recall the Cartan decomposition  
$\displaystyle \Gpp=\bigsqcup_{\lambda\in\mathbb X_\bullet^+} \Gbb t^{\lambda}\Gbb$. Note that gauge transfermation by  elements of $\Gbb$ will not change the order of singularity and the non-nilpotentnece of the polar part; therefore, we are reduced to showing that the order of singularity of $t^\lambda\cdot\nabla$ is $\geq n$.
 Let $A_{-n}=A_{\bar\fn}\oplus A_\ft\oplus A_\fn$
be the decomposition induced by the triangular decomposition $\fg=\bar\fn\oplus\ft\oplus\fn$. If $A_\ft\neq 0$, then 
the operator $t^\lambda\cdot\nabla$ will contain the summand $A_\ft/t^n$ hence the order 
of singularity of $t^\lambda\cdot\nabla$ is at least $n$. If $A_\ft=0$, then since $A_{-n}$ is non-nilpotent it 
implies $A_{\bar\fn}\neq 0$. Let $A_{\bar\fn}=\oplus_{\alpha\in\Delta^+} A_{\bar\fn,-\alpha}$, then
any nonzero $A_{\bar\fn,-\alpha}$ will contribute a summand $A_{\bar\fn,-\alpha}t^{-n-<\lambda,\alpha>}$
in $t^\lambda\cdot\nabla$. Hence, the oder of singularity of $t^\lambda\cdot\nabla$ is $\geq n$.
\end{proof} 

The above lemma motivates the following definition. 
\begin{defe} \label{d:reduced}
The operator \eqref{eq:op} is in the \emph{reduced form} if $A_{-n}$ is not nilpotent.
\end{defe} 

It is not always possible to put a connection in a reduced form using $\Gpp$-gauge transformation. This is possible, however, if we allow ourselves to go to the field extension $\bC\llp t^{1/b}\rrp$ for some positive integer $b$. This is the content of the following lemma, which was originally proved in \cite{Deligne70} and \cite{Katz87} for $G=\mathrm{GL}_n$ and \cite{Varadarajan} for general $G$. We give a proof in \S\ref{ss:slopeOp}, using opers.

\begin{lem} Let $\sigma$ be a irregular local system on $\cDt$. Then there exists a positive integer $b$ such that the $G\llp t^{1/b} \rrp$-gauge equivalence class of $\sigma$ contains an operator in the reduced form. 
\end{lem}

Next, observe that if we set $u=t^{1/b}$, then we can write the operator \eqref{eq:op} as 
\[
\nabla=A_{-n} u^{-nb} + \cdots 
\]
This motivates the  definition of slope:

\begin{defe}[cf. \cite{FGross}] The slope $s(\sigma)$ of a local system $\sigma$ is defined as follows: $s(\sigma)=0$ if $\sigma$ is regular singular; otherwise, 
 $s(\sigma)=a/b$ if the connection $\sigma$ is $G \llp t^{1/b} \rrp$-gauge equivalent to a reduced operator with order of singularity $a+1$. 
\end{defe}

\begin{lem}\label{l:Deligne}
The definition of slope is well defined.
\end{lem}
\begin{proof}
We can assume $\sigma$ is irregular.
Let $l$ be another positive integer such that $\nabla$ is $G\llp t^{1/l}\rrp$-gauge equivalent to
a reduced operator with a pole of order $k+1$.
We have to show $a/b=k/l$. 
Passing to the extension $\bC(t^{1/bl})$, we see that the connection $\nabla$ is 
$G\llp t^{1/bl}\rrp$-gauge equivalent to both $\nabla_1$ and $\nabla_2$, where 
$\nabla_1$ (resp. $\nabla_2$) is a reduced operator having a pole of order $la+1$ (resp. $bk+1$). 
We claim that 
\[
la+1=bk+1.
\]
 Clearly, this implies  $a/b=k/l$ hence finished the proof the lemma.
Now observe that the operators $\nabla_1$ and $\nabla_2$ are $G\llp t^{1/bl}\rrp$-gauge 
equivalent. Thus, the claim follows from lemma \ref{order of singularity}. 
\end{proof}

\begin{rem}\label{r:vb} Let $\sigma=(\cF,\nabla)$ be a local system on a bundle $\cF$. 
Let $\rho:G\ra \mathrm{GL}(V)$ be a faithful representation of $G$, and let $\sigma^V$ be the induced 
connection on the associated vector bundle. Since being reduced is preserved under $\rho$, we see that 
the slope of $\sigma$ is equal to the slope of the connection $\sigma^V$.
\end{rem}

%%%%%%%%%%%%%%%%%%%%%%%%%%%%%%%%%%%%%%%%%%%%%%%%%%%%%%%%%%%%%%%%%%%%%%%%

\ssec{Tannakian formulation} \label{ss:slopeTan}
Let us first recall the definition of differential Galois group of $\cDt$
following \cite[\S 2]{Katz87}. Let $\on{Conn}(\cD^\times)$ be the 
category of connections on $\cD^\times$. By definition, the objects of this category are pairs $(V,\nabla)$ consisting of a finite dimensional vector space $V$ over $\bCpp$ and a connection $\nabla$ on $V$. Note that if $\dim(V)=n$, then $(V,\nabla)$ is an element of $\Loc_{\mathrm{GL_n}(V)}(\cDt)$. The category $\on{Conn}(\cDt)$ has a natural notion of internal Homs and tensor products, giving it the structure of a rigid abelian tensor category with $\End_\bone=\bC$. It has, moreover, an evident $\bCpp$-valued fibre functor; namely, the functor which sends the pair $(V,\nabla)$ to $V$. 

Using the results of Levelt, Katz constructed a canonical $\bC$-valued fibre functor  $F: \on{Conn}(\cDt)\ra \mathrm{Vect}_\bC$. Thus, he showed that $\on{Conn}(\cDt)$ is, in fact, a \emph{neutral} Tannakian category over $\bC$. The differential 
Galois group $I=\on{Aut}(F)$ is the group of automorphisms of this fibre functor. 
It is a pro-algebraic group over $\bC$
whose finite dimensional representations are identified with objects of $\on{Conn}(D^\times)$.

The group $I$ has an ``upper numbering filtration" defined as follows. 
For every nonnegative real number 
$r$, let $I^r$ to be the kernel of 
$I\ra\on{Aut}(F|_{\on{Conn}_<r}(D^\times))$, 
where $\on{Conn}_{<r}(D^\times)$ 
is the subcategory of connections of slope
less than $r$. Similarly, one defines
$I^{r^+}$ to be the kernel of 
$I\ra\on{Aut}(F|_{\on{Conn}_{\leq r}}(D^\times))$.
For every $0<x<y$ we have 
\[
I^y\subset I^{x^+}\subset I^x\subset I.
\]

We claim that the data of a $G$-local system on $\cDt$ is the same as a homomorphism $I\ra G$. 
Indeed, suppose we are given $(E,\nabla)\in\on{Loc}_G(D^\times)$.
Then the  induction functor $E\ra E\times^GV$
defines a tensor functor
\[
\nabla:\on{Rep}(G)\ra\on{Conn}(D^\times).
\]
Composing with Katz's fibre functor, we obtain a tensor functor $\on{Rep}(G)\ra  \on{Vect}_\bC$, which by Tannaka duality, gives us a homomorphism
$\phi_\nabla:I\ra G$. The converse construction is also evident.

We now turn our attention to defining the slope of a local system using this formalism. 

\begin{defe}
Let $\sigma=(E,\nabla)\in\Loc_G(\cD^\times)$.
Define
\[
s'(\sigma)=\on{inf}\{r\geq 0\, \, |\, \,  I^{r^+}\subset
\on{ker}(\phi_\nabla)\}
\]
\end{defe}
\begin{lem} For all $\sigma \in \Loc_G(\cDt)$, 
we have $s'(\sigma)=s(\sigma)$.
\end{lem}
\begin{proof}
Let $(r,V)$ be a faithful representation of $G$ and let $\sigma^V$  be the 
induced connection.
Then from the definition of $s'$ we see that 
$s'(\sigma)$ is equal to $s'(\sigma^V)$. By Remark \ref{r:vb}, we are reduced to the case of vector bundles and the lemma 
follows form the definition of the upper numbering filtration group. 
\end{proof}

%%%%%%%%%%%%%%%%%%%%%%%%%%%%%%%%%%%%%%%%%%%%%%%%%%%%%%%%%%%%

\ssec{Recollections on Moy-Prasad Theory} 
Let $G$ be a connected reductive group over $\bC$, and let $Z$ denote the centre of $G$. Let $\fg$ denote the Lie algebra of $G$. We fix a maximal torus $T\subseteq G$ with the corresponding Cartan subalgebra $\ft$. Let $\Phi$ denote the set of roots of $G$ with respect to $T$. For ease of notation, we set 
\[
\Phi^*=\Phi \sqcup \{0\}
\]
 For $\alpha \in \Phi^*$ let $\fu_\alpha\subset \fg$ denote the weight space for $T$ corresponding to $\alpha$ (so $\fu_0=\ft$). Let $\Gpp$ and $\fgpp$ denote the corresponding loop group and loop algebra. 

Let $\ocB$ be the Bruhat-Tits building of $\Gpp$ and let $\cB=\ocB\times (X_*(Z)\otimes_\bZ \bR)$ denote the enlarged building. Let $\cA=\cA(G,T)$ denote the standard apartment of $\cB$. This is an affine space isomorphic to $X_*(T)\otimes_{\bZ} \bR$.  We have an isomorphism $\cA\simeq \ft_{\bR}$; thus, points in $\cA$ may be viewed as elements of $\ft$. 

For every $x\in \cB(G)$ and $r\in \bR_{\geq 0}$, Moy and Prasad defined subgroups $G_{x,r}$ and $G_{x,r^+}$ inside $\Gpp$. In the case $r=0$, then $G_x:=G_{x,0}$ is the usual parahoric subgroup associated to $x$ and $G_{x^+}:=G_{x,0^+}$ is the pro-unipotent radical of $G_x$. Recall that $\Gpp$ acts on $\cB(G)$. If $g.x=y$, for $g\in \Gpp$ and $x,y\in \Gpp$, then $\on{Ad}(g).G_{x,r}=G_{y,r}$. Dito for $G_{x,r^+}$ and its Lie algebra. Since, every $x\in \cB(G)$ has an element of $\cA$ in its orbit, in the applications we have in mind, it suffices to consider $x\in \cA$. 

It will be convenient for us to have an explicit description of the Lie algebras of $G_{x,r^+}$, where $x\in \cA$. This is given by 
\begin{equation}\label{eq:gxr+}
\fg_{x,r^+}=\bigoplus_{\alpha\in \Phi^*} \fu_\alpha(\cP^{1-\lceil \alpha(x)-r \rceil}).
\end{equation} 
Here, $\cP=t\bCbb$ denotes the maximal ideal of $\bCbb$. 

Let $\fg^*$ denote the dual of $\fg$. Moy and Prasad also defined filtration subalgebras of $\fg^*_{x,r}$ and $\fg^*_{x,-r}$, where now $r\in \bR_{\leq 0}$.  
One has a canonical isomorphism 
\begin{equation} 
(\fg_{x, r} / \fg_{x, r+})^* \simeq \fg^*_{x, -r} /
\fg^*_{x, -r+}.
\end{equation}

 Finally, let us make some remarks regarding optimal points. 
Fixed a chamber $C\subset\mathcal A$.
Let $O\subset\overline C$ denote the set of optimal points (see \cite{MP} for the definition 
of optimal points). The set $O$ has may good properties. For example,
the set $\big\{r\in\bR_{\geq 0}|G_{x,r}\neq G_{x,r^+}\big\}$ is a discrete 
subset of $\mathbb Q$. Elements in above set are called optimal numbers.
Also, for any $(y,r)\in\mathcal A\times\bR_{\geq 0}$, there are 
$x,z\in O$ such that 
\begin{equation}\label{optimal}
G_{x,r^+}\subset 
G_{y,r^+}\subset G_{z,r^+}.
\end{equation} 
In many applications of Bruhat-Tits and Moy-Prasad Theory, it is enough to consider the optimal points of the building, as oppose to arbitrary points.

\ssec{Slopes via Moy-Prasad Theory} 
We are now ready to give Bremer and Sage's definition of the slope \cite{BS}.  
 Let $(\cF, \nabla)$ be a pair consisting of a $G$-bundle $\cF$ on $\cDt$ equipped with a connection $\nabla$. 
 Choosing a trivialisation $\phi$ of $\cF$, we can write $\nabla$ in terms of a one-form with coefficients in $\fgpp$. We denote this one-form by 
$[\nabla]_\phi \in \Omega^1(\fgpp)$. Recall that a point $x\in \cA$ defines an element in $\ft$ which, by an abuse of notation, is also denoted by $x$. Therefore, $x\frac{d}{dt}$ is an element of $\Omega^1(\fg)\subset \Omega^1(\fgpp)$ and so $[\nabla]_\phi-x\frac{d}{dt}$ makes sense as an element of $\Omega^1(\fgpp)$. Now the residue pairing defines a canonical isomorphism 
\[
\Omega^1(\fgpp)\simeq \fg^*\llp t \rrp.
\]
 Thus, we may think of the one-form $[\nabla]_\phi-x\frac{d}{dt}$ as an element of $\fg^*\llp t\rrp$. Recall that $\fg^*_{x,-r}$ is a lattice inside $\fg^*\llp t \rrp$ for $r\in \bR_{\geq 0}$. 

A stratum is a triple $(x,r,\beta)$ consisting of a point $x\in \cA(G)$, a number $r\in \bR_{\geq 0}$, and a functional $\beta\in (\fg_{x,r}/\fg_{x,r^+})^*$. 

\begin{defe} We say that the flat $G$-bundle $(\cF,\nabla)$ contains the stratum $(x,r,\beta)$ with respect to the trivialisation $\phi$ if $[\nabla]_\phi-x\frac{d}{dt} \in \fg^*_{x,-r}$ and the coset $([\nabla]_\phi-x\frac{dt}{t})+ \fg^*_{x,-r^+}$ equals the coset determined by the functional 
\[
\beta\in (\bar{\fg}_{x,r})^*=(\fg_{x,r}/\fg_{x,r^+})^*\simeq \fg^*_{x,-r}/\fg^*_{x,-r^+}.
\]
\end{defe} 

It is proved in \cite{BS} that every $\sigma$ contains a stratum. In particular, the following definition make sense. 
\begin{equation}
s_{BS}(\sigma):= \min \{ r\in \bR_{\geq 0} \, \, |\, \,   \textrm{$\sigma$ contains a stratum of the form  $(x,r,\beta)$} \}
\end{equation}

Bremer and Sage establish the following result. 

\begin{thm} For every $\sigma\in \Loc_G(\cDt)$, we have $s_{BS}(\sigma)=s(\sigma)$. 
\end{thm}

\begin{rem}\label{r:Yu} Bremer and Sage prove that the slope may also be characterised as 
the minimum depth of a triple $(x,r, \beta)$ contained in $(\cF,\nabla)$ for which $x$ is an optimal point. Apparently, J. K. Yu has proved that the denominators of the critical numbers at the optimal points are divisors of the fundamental degrees of $\fg$ (unpublished). From these considerations, it follows that the denominator of the slope of a flat connection  divides a fundamental degree of $\fg$. This fact will be evident from our definition of slope via opers. 
\end{rem}

\ssec{Recollections on opers}\label{ss:opers}
Let $G$ be a reductive group of rank $\ell$. We fix a Borel subgroup $B$ and let $N=[B,B]$ and $T=B/N$, and let $W$ denote the Weyl group of $G$ and let $Z$ be the center of $G$.  Let $\fg$, $\fb$, $\fn$, and $\ft$ denote the corresponding Lie algebras. Choose generators $f_i$ for the root subgroups corresponding to the negative simple roots of $\fg$. Let 
\[
\displaystyle p_{-1}=\sum_{i=1}^{\ell} f_i.
\]

The notion of opers is due to Beilinson and Drinfeld \cite{BDOper}, building on earlier work by Drinfeld and Sokolov \cite{DS}. For us, the following description, given in terms of a local coordinate, suffices: a $G$-oper (on the punctured disk $\cDt$) 
is an $B\llp t\rrp$-gauge equivalence class of operators of the form 
\begin{equation} \label{eq:operDt}
\nabla =\partial_t + \sum\phi_if_i + v,\quad \quad \phi_i\in\bC\llp t\rrp^\times, \quad v \in \fb \llp t \rrp.
\end{equation} 
Let $\Op_G(\cDt)$ denote the set of $G$-opers on $\cDt$.
We will now give an explicit description of this set.

 Let $p_1$ denote the unique element of degree 1 in $\fn$ such that $\{p_{-1}, 2\rho, p_1\}$ is an $\Sl_2$-triple. Let 
\[
V_\can:=\bigoplus_{i\in E} V_{\can, i}
\]
denote the $\ad\, p_1$-invariants in $\cfn$, decomposed according to the principal gradation. Here $E=\{d_1, \cdots, d_\ell\}$ is the set of exponents of $\cfg$. As mentioned in \S 1.3 of \cite{FG}, it follows from a theorem of Kostant (cf. \cite{Drinfeld}) that the composition 
\begin{equation} \label{eq:Kostant} 
V_\can \rar {v\mapsto v+p_{-1}} \fg\ra \fg\quo G \simeq \ft\quo W
\end{equation}  
is an isomorphism. In particular, $V_\can$ is an affine space of dimension $\ell$.  The following lemma is due to Drinfeld and Sokolov \cite{DS}; see also Lemma 4.2.2 of \cite{FrenkelBook}.

\begin{lem}\label{I:DS} 
\begin{enumerate} 
\item[(i)] Every $B\llp t\rrp$-gauge equivalence class of 
operators \eqref{eq:operDt} contains an operator of the form
\begin{equation} \label{eq:canForm}
\nabla = \partial_t + p_{-1} + v,
\end{equation}
where $v\in V_\can\llp t\rrp$.
\item[(ii)] If $G$ is adjoint, the conjugation action of $B\llp t\rrp$ on the space of operators of the form \eqref{eq:operDt} is free and each conjugacy class contains a unique operator of the form in (\ref{eq:canForm}).
 \end{enumerate}
\end{lem} 

The expression \eqref{eq:canForm} is called the \emph{canonical form} of the oper. The above lemma implies that every oper has a canonical form. If 
$G$ is adjoint it has a unique canonical form and
we have an isomorphism $\Op_G(\cDt) \simeq V_\can \llp t\rrp$. In fact, this is an isomorphism of ind-schemes of ind-finite type over $\bC$. However, this isomorphism is not canonical since it depends on a choice of coordinate $t$.

\ssec{Slope via opers} \label{ss:slopeOp}

\sssec{} Let $\chi\in \Op_G(\cDt)$ and write $\chi$ in a canonical form $\partial_t+p_{-1}+v$, where $v\in V_\can\llp t\rrp$. Choose linear generators $p_j\in V_{\can, d_j}$ (if the multiplicity of $d_j$ is greater than one, then we choose linearly independent vectors in $V_{\can, j}$), and write
\[
 v(t) = \sum_{i=1}^\ell v_j p_j, \quad \quad v_j\in \bC\llp t \rrp. 
 \]
 Let us write $v_j=t^{-n_j}.h_i$ where $h_i\in \bCbb^\times$. 
In the introduction, we defined the slope of $\chi$ is defined by 
\[
s(\chi):= \sup \{ 0, \sup{n_i/(d_i+1)-1}_{i=1,\cdots, \ell} \}. 
\]

Recall that Proposition \ref{p:slopeOper} states that the slope of an oper equals the slope of its underlying connection. 

\begin{proof}[Proof of Proposition \ref{p:slopeOper}]
Let $\sigma=(\mathcal F,\nabla)$ be the underling connection of the oper $\chi$.
Let $\nabla=\partial_t+p_{-1}+v$ be a canonical form coming from the oper structure. 
Write $v_i=t^{-n_i}\cdot h_k$ where $h_k\in\bCbb^\times$. Assume that $n_k/(d_k+1)$ is maximal among all
$\{n_i/(d_i+1)\}_{i=1,...,r}$ . Now take the covering $s=t^{1/(d_k+1)}$, the connection becomes
\[\nabla=d+(p_{-1}+s^{-n_k(d_k+1)}h_kp_k+\oplus_{i\neq k}s^{-n_k(d_i+1)+c_i}h_ip_i)(d_k+1)s^{d_k}\]
where $c_i\in\bZ_{+}$. Conjugating with $s^{n_k\cdot\check\rho}$ we get 
\[\nabla=d+(d_k+1)s^{d_k-n_k}(p_{-1}+h_kp_k+\oplus_{i\neq k}s^{c_i}h_ip_i)+n_k\check\rho s^{-1}.\]
If $s(\chi)=0$, then we have $n_k-d_k\leq 1$, thus the connection has regular singularity and it 
implies $s(\sigma)=0$.
If $s(\chi)>0$, then we have $n_k-d_k>1$. Since $h_k=z_0+z_1s+\cdot\cdot\in\bC\llb s\rrb^\times$ and $c_i\in\bZ_+$, 
we see that 
the order of the singularity of the connection $\nabla$ is $n_k-d_k$ and the polar part is $p_{-1}+z_0p_k$ which is not nilpotent
(by the theorem of Kostant). Hence from the definition of slope in \S\ref{ss:slopeDef}, we have 
\[s(\sigma)=(n_k-d_k-1)/(d_k+1)=n_k/(d_k+1)-1=s(\chi).\] 
\end{proof}

\begin{rem} 
A flat $GL(n)$-bundle on $\mathcal D^\times$ is the same as a
rank $n$ bundle $E$ on $\mathcal D^\times$ with a flat connection $\nabla$. 
An oper structure on $(E,\nabla)$ is equivalent to the existence of cyclic vector
of $(E,\nabla)$. The existence of cyclic vector is proved in \cite{Deligne70}.
Furthermore, in \cite[p. 49]{Deligne70} Deligne give a definition of slope using the cyclic vector. One can show that his definiton is equivalent to 
our Definition \ref{slope via oper} in the case $G=GL(n)$.
\end{rem}

Recall the definition of reduced form of an operator (Definition \ref{d:reduced}). 
\begin{cor}
Every irregular flat $G$-bundle on $\mathcal D^\times$ is $G((t^{1/b}))$-gauge equivalent to 
a reduced operator for some $b\in\bZ_{>0}$.
\end{cor}
\begin{proof}
Indeed, during the proof of the above proposition we showed that every oper
whose underling flat $G$-bundle is irregular is $G((t^{1/b}))$-gauge equivalent to 
a reduced operator. Since every flat $G$-bundle has an oper structure, result follows.
\end{proof}

\begin{exam}[cf. \cite{FGross}]
Let $G$ be a simple group of adjoint type.
Let us compute the slope of the operator 
\[
\nabla=\partial_t+\frac{p_{-1}}{t}+\frac{p_k}{t^2}.
\]
Recall that $p_{-1}=\sum_{i=1}^{l} f_i$ and $p_k\in V_{can,k}$.
Let us write $m=d_k+1$.
Passing to the extension $s=t^{1/m}$, the connection becomes 
\[\partial_t+\frac{mp_{-1}}{s}+\frac{mp_k}{s^{m+1}}.\]
Gauge transforming this operator with $g=\rho(s)$, where $\rho$ is the  half-sum pf
positive coroots, the connection becomes
\[
\partial_t+\frac{m(p_{-1}+p_{k})}{s^2}-\frac{\rho}{s}.
\] 
By a theorem of Kostant, the element $p_{-1}+p_k$ is non-nilpotent; hence, the slope is equal to $\frac{1}{m}$.
Alternatively, we compute the slope of $\nabla$ using the canonical form of oper. Conjugating by $t^{-\rho}$, the connection 
$\nabla$ becomes
\[
\partial_t+p_{-1}+\frac{\rho}{t}+\frac{p_k}{t^{d_k+2}}.
\]
To get rid of $\rho\cdot t^{-1}$ we conjugate by $\on{exp}(\frac{-p_1}{2t})$ and obtain the operator 
\[
\partial_t+p_{-1}-\frac{p_1}{4t^2}+\frac{p_k}{t^{d_k+2}}.
\]
This operator has the form of an oper. Now using definition \ref{slope via oper}, one can easily see that 
its slope (as an oper) equals $\frac{d_{k}+2}{d_{k}+1}-1=\frac{1}{m}$.
\end{exam}

%%%%%%%%%%%%%%%%%%%%%%%%%%%%%%%%%%%%%%%%%%%%%%%%%

\section{Representations of affine Kac-Moody algebras} \label{s:rep}
The main purpose of this section is to proof Theorem \ref{t:support'}. In the first subsection, we recall the definition of Moy-Prasad subalgebras and define the notion of depth for smooth modules for affine Kac-Moody algebras. The results of this subsection is not used in the rest of the paper. In \S \ref{ss:vertex}, we collect some basic information regarding vertex operators.  In \S \ref{ss:smooth}, we apply these general considerations to  affine vertex algebras. We recall some basic properties of Segal-Sugawara vectors and operators in \S \ref{ss:Sugawara}. This is useful because according to \cite[\S 3]{FrenkelBook}, Segal-Sugawara operators can be interpreted as elements in the centre of the completed enveloping algebra. Armed with these preliminaries, we prove Theorem \ref{t:support'} in \S \ref{ss:proof}.

\ssec{Depths of smooth modules} \label{ss:depth} Let $\fg$ be a simple Lie algebra over $\bC$.
Let $\kappa$ be an invariant bilinear form on $\fg$. The affine Kac-Moody algebra $\hfg_\kappa$ at level $\kappa$ is defined to be the central extension
\begin{equation}\label{eq:central}
0\ra \bC.\bone \ra \hfg_\kappa\ra \fgpp \ra 0,
\end{equation} 
with the two-cocycle defined by the formula 
\begin{equation}\label{eq:2cocycle}
(x\otimes f(t), y\otimes g(t)) \mapsto -\kappa(x,y).\,\Res_{t=0} fdg.
\end{equation}

A module $V$ over $\hfg_\kappa$ is \emph{smooth} if for every $v\in V$ there exits $N_v\geq 0$ such that $t^{N_v} \fgbb$ annihilates $v$, and such that $\bone\in \bC.\bone \subset \hfg_\kappa$ acts on $V$ as the identity.  Thus, every vector in a smooth module $V$ is annihilated by a bounded subalgebra. The notion of depth measures, in some sense, the largest bounded subalgebra which annihilates a vector in $V$. 

 It will be convenient to have the following notation. Let $\Phi$ denote the set of roots of $\fg$ with respect to a Cartan subalgebra $\ft$, and let  $\Phi^*:=\Phi\sqcup \{0\}$. Then we have the root decomposition 
 \begin{equation}\label{eq:decomposition} 
 \fg= \bigoplus_{\alpha \in \Phi^*} \fg^{\alpha}, 
 \end{equation} 
 where by definition $\fg^0=\ft$. Recall the description of Moy-Prasad subalgebras $\fg_{x,r^+}\subset \fgpp$ given in \eqref{eq:gxr+}.

 \begin{lem}\label{l:split}
  For all $(x,r)\in \cB(G) \times \bR_{\geq 0}$, the central extension \eqref{eq:central} is split over $\fg_{x,r^+}$. 
\end{lem}

\begin{proof}
It is enough to prove the lemma for $x\in \mathcal A$ and  $r=0$. Consider the two-cocycle \eqref{eq:2cocycle}. Suppose $x\otimes f(t)$ and $y\otimes g(t)$ are in $\fg_{x}$ and $x\in \fg^\alpha$. Then if we want $\kappa(x,y)\neq 0$ then we must have $y\in \fg^{-\alpha}$. The case of $\alpha=0$ is immediate, so assume $\alpha\neq 0$.Then, we have that $f\in \cP^{1-\lfloor \alpha(x) \rfloor}$ and $dg\in \cP^{-\lfloor -\alpha(x) \rfloor}$. It follows that $fdg$ is regular; hence, its residue is zero. 
\end{proof} 

\begin{defe} Let $V\in \hfg_\kappa-\modd$. Define 
\[
d(V):=\inf \{ r\in \bR_{\geq 0} \, | \, \exists \textrm{$x\in \cB(G)$ such that $V^{\fg_{x,r^+}}$ is non-empty} \}. 
\]
\end{defe} 

Using optimal points, one can show that the depth of every smooth module is a rational number.

\begin{rem} \label{r:intDepth} 
If $V$ is $G_{x,r^+}$ integrable, then $V^{\fg_{x,r^+}}$ is non-empty. Indeed, since $V$ is smooth, there exist $s>r$
such that $W^{\fg_{x,s}}$ is non-empty. Let $w_0\in W^{\fg_{x,s}}$ be any nonzero vector
and consider the vector sub-space $W_0=U(\fg_{x,r^+})\cdot w_0\subset W$. 
The space $W_0$ carries a $\fg_{x,r^+}$-action and this action factors through the 
quotient $\fg_{x,r^+}/\fg_{x,s}$.  The action of $\fg_{x,r^+}/\fg_{x,s}$
integrates to an action of $G_{x,r^+}/G_{x,s}$. The latter is a finite dimensional unipotent algebraic group; therefore,
by Kolchin's Theorem, $W_0$ contains a non-zero $G_{x,r^+}$-invariant  vector. Thus,
$W^{G_{x,r^+}}=W^{\fg_{x,r^+}}$ is non-empty. 
\end{rem}

 \begin{lem}  \label{c:depthV}
  Let $V\in \hfg-\modd$ be a smooth module at the critical level with depth $r$. Then the central support of $V$ is a subscheme of $\Op_\hG^r$. 
\end{lem}

\begin{proof} By assumption, there exists $x\in \cB(G)$ such that $V$ has a vector $v$ annihilated by $\fg_{x,r^+}$. By the previous remark, we have a nontrivial morphism $\bU_{x,r} \ra V$, sending the generating vector of $\bU_{x,r}$ to $v$.  The result follows from  Theorem \ref{t:support}. 
\end{proof}

Let $V$ be a smooth module of depth $r$ at the critical level admitting a central character $\chi\in \Op_\hG(\cDt)$. The previous lemma states that $s(\chi)\leq r$. We don't expect, in general, that $s(\chi)=r$. For instance, there are regular opers $\chi$ such that the central reduction $\bW_\Psi(\chi)$ of the affine Whitaker module with respect to $\phi$ is non-zero; see \S \ref{ss:global}. In this case, the depth $\bW_\Psi(\chi)$ is $1/h$, where as $s(\chi)=0$. This is one reason the notion of depth in the categorical setting is better behaved: we expect equality in that case (Conjecture \ref{c:main}).

 \ssec{Recollections on vertex algebras} \label{ss:vertex}
 We make some general (and obvious) observations about fields in vertex algebras. 
  Let $V$ be a vertex algebra. For $A\in V$, let 
\[
A(z)=Y(A,z)=\sum_m A_m z^{-m-1}
\]
be the corresponding field. Recall that $A_m\in \End(V)$. 
We write $[A(z)]_m$ for the $m$th Fourier coefficient of a field $A(z)$; that is, $[A(z)]_m=A_m$. 
For example, consider the normally ordered product 
\[
A(z)B(z)=\sum_{s \in \bZ} \left(   \sum_{r<0} A_r B_sz^{-r-1} + \sum_{m\geq 0} B_s A_r z^{-r-1} \right) z^{-s-1}
\]
The coefficient of $z^{-m-2}$ is equal to 
\[
\sum_{r+s=m, r<0} A_r B_s + \sum_{r+s=m, r\geq 0} B_s A_r. 
\]
We conclude that the coefficient $[A(z)B(z)]_m$ is a linear combination of the elements $A_r B_s$ or $B_sA_r$ where $r+s=m-1$. The following lemma is an obvious generalisation of this statement.

\begin{lem} 
If $A_1(z), \cdots, A_k(z)$ are fields, then $[:A_1(z)\cdots A_k(z):]_m$ can be written as a linear combination of elements of $[A_{\sigma(1)}(z)]_{m_1}\cdots [A_{\sigma(k)}(z)]_{m_k}$, where $\sigma$ is a permutation of $\{1,\cdots, k\}$ and $m_1+\cdots m_k = m-(k-1)$. 
\end{lem}

 We apply the above considerations to the fields for the affine vertex algebra $V_c(\fg)$ at the critical level. Let $x\in \fg$ and suppose $n<0$. Then, in view of the isomorphism of vector spaces $V_c(\fg)\simeq U(t^{-1}\fg[t^{-1}])$, we can think of $x_n=x\otimes t^n$ as an element of $V_c(\fg)$. The field corresponding to this element is 
\[
x_n(z)=\frac{1}{(-n-1)!} \partial_z^{-n-1} \sum_{m\in \bZ} x_m z^{-m-1}.
\]
The Fourier coefficients can then be considered as elements of the completed universal enveloping algebra $\tU_c(\hfg)$ at the critical level, cf. \cite[\S 3]{FrenkelBook}.
Now, we have that $[x_n(z)]_m$ is a multiple of $x_{m+n+1}$. Thus, if $n=-1$, then $[x_{-1}(z)]_m=x_m$. 

Next, suppose $x^{1}, \cdots, x^{k}$ are  elements of $\fg$. Let $x=x_{n_1}^{1}\cdots x_{n_k}^{k}$ where  $n_j<0$ for all $j$. As above, we can think of $x$ as an element of $V_c(\fg)$. The corresponding field $x[z]$ is the normally ordered product $:x_{n_1}^{1}(z)\cdots x_{n_k}^{k}(z):$. Write $x_{[m]}=[x(z)]_m$ for the $m$th Fourier coefficient of the field $x[z]$. Then by the previous lemma, $x_{[m]}\in \tU_c(\hfg)$ is a linear combination of elements of the form
\[
[x_{n_{\sigma(1)}}^{{\sigma(1)}}(z)]_{m_1} \cdots [x_{n_{\sigma(k)}}^{{\sigma(k)}}(z)]_{m_k}.
\]

Thus, we obtain:

\begin{cor} The operator $x_{[m]}$ is a linear combination of monomials of the form 
\[
x_{m_1+n_{\sigma(1)}+1}^{{\sigma(1)}}\cdots x_{m_k+n_{\sigma(k)}+1}^{{\sigma(k)}},
\]
where $\sigma$ is some permeation of $\{1,\cdots, k\}$ and $m_i$'s are integers with $m_1+\cdots m_k = m-(k-1)$. 
\end{cor}

 \subsection{Fourier coefficients acting on smooth modules} \label{ss:smooth}  
  Recall that convention of \eqref{eq:decomposition}. We write $x^\alpha$ for an element of $\fg^{\alpha}$ and $x_n^\alpha$ for the element $x^\alpha \otimes t^n$ in $\hfg_c$.  
 
 \begin{lem} \label{l:annihilate} 
 Let $V$ be a $\hfg_c$-module and let $v\in V$. For $\alpha \in \Phi^*$, let $r_\alpha\in \bZ$ be such that 
 \begin{enumerate} 
 \item $x_n^\alpha. v =0$, for all $n\geq r_\alpha$, and 
 \item  $r_\alpha + r_\beta \geq r_{\alpha+\beta}$. 
 \end{enumerate}

 Let $x=x_{n_1}^{\alpha_1}\cdots x_{n_k}^{\alpha_k}\in \tU_c(\hfg)$ and suppose $\sum_{j=1}^k n_j \geq \sum_{j=1}^k r_{\alpha_j}-k+1$. Then $x.v=0$. 
 \end{lem}  
 
 \begin{proof} We apply induction on $k$. The result is clearly true for $k=1$. Now if $n_{\alpha_k}\leq r_{\alpha_k}$ for all $k$, we would have that $\sum_{j=1}^k n_j \leq \sum_{j=1}^k r_{\alpha_j}-k$, a contradiction. Thus, we have that $n_{\alpha_j}\geq r_{\alpha_j}$ for some $j$.
If $j=k$, then we have that $n_{\alpha_k}\geq r_{\alpha_k}$, and so $x_{n_k}^{\alpha_k}$ (and therefore $x$) annihilate $v$.  Otherwise, we move $x_{n_j}^{\alpha_j}$ to the right, using the equality 
 \begin{equation}\label{eq:sum}
 x_{n_1}^{\alpha_1} \cdots (x_{n_j}^{\alpha_j} x_{n_{j+1}}^{\alpha_{j+1}}) \cdots x_{n_k}^{\alpha_k}=
  x_{n_1}^{\alpha_1} \cdots (x_{n_{j+1}}^{\alpha_{j+1}} x_{n_j}^{\alpha_j})  \cdots x_{n_k}^{\alpha_k} + 
   x_{n_1}^{\alpha_1} \cdots [x_{n_j}^{\alpha_j}, x_{n_{j+1}}^{\alpha_{j+1}}] \cdots x_{n_k}^{\alpha_k}
 \end{equation}
 But note that $[x_{n_j}^{\alpha_j}, x_{n_{j+1}}^{\alpha_{j+1}}]=x_{n_j+n_{j+1}}^{\alpha_{j}+\alpha_{j+1}}$. Moreover, by assumption 
 \[
 (n_1 + \cdots + n_k) \geq r_{\alpha_1} + \cdots r_{\alpha_j} + r_{\alpha_{j+1}} +\cdots +  r_{\alpha_n}
  \geq r_{\alpha_1} + \cdots +r_{\alpha_j+\alpha_{j+1}} +\cdots +  r_{\alpha_n}
 \]
 Therefore, by induction, the second summand of the above annihilates $v$. Hence, we obtain 
 \[
  x_{n_1}^{\alpha_1} \cdots (x_{n_j}^{\alpha_j} x_{n_{j+1}}^{\alpha_{j+1}}) \cdots x_{n_k}^{\alpha_k}.v=
  x_{n_1}^{\alpha_1} \cdots (x_{n_{j+1}}^{\alpha_{j+1}} x_{n_j}^{\alpha_j})  \cdots x_{n_k}^{\alpha_k}.v
 \]
 Repeating the above procedure if necessary, we can move  $x_{n_j}^{\alpha_j}$ to the right-most position and annihilate $v$.  
  \end{proof}

\begin{cor} \label{l:field} 
Suppose we are in the set up of the previous lemma, and assume $x=x_{n_1}^{\alpha_1}\cdots x_{n_k}^{\alpha_k}$ is in $V_c(\hfg)$; that is, assume all $n_i<0$. Then, we have 
\[
x_{[m]}.v=0 \quad \quad \forall \, m\geq \sum_{i=1}^{k} ({r_{\alpha_i}} -n_i)-k.
\]
\end{cor}

\begin{proof} We know that 
$x_{[m]}$ is a linear combination of elements of the form 
\[
x_{m_1+n_{\sigma(1)}+1}^{\alpha_{\sigma(1)}}\cdots x_{m_k+n_{\sigma(k)}+1}^{\alpha_{\sigma(k)}},
\]
where $m_1+\cdots m_k = m-k+1$. Now observe that 
\[
\sum_{i=1}^k (m_i + n_{\sigma(i)}+1) = \sum_{i=1}^k m_i + \sum_{i=1}^k n_{\sigma(i)} + k = m+1 + \sum_{i=1}^k n_i 
\]
The assumption on $m$ implies that the above sum is greater than or equal to $\sum r_{\alpha_i}-k+1$. 
The result follows from the previous lemma. 
 \end{proof}

\begin{exam} Suppose we are in the situation of the corollary and $r_\alpha=n$ for all $\alpha$. In other words, $v$ is killed by $t^n\fgbb$.  Let $x$ be an element of degree $N$; that is, $N:=-\sum_{i} n_i\geq k$. Then the previous corollary implies that 
 \[
 x_{[m]}.v=0,\quad \quad \forall \, m\geq k.n+N-k=k(n-1)+N.
 \]
  Note that we always have $k\leq N$. Now suppose $n\geq 1$. Then $N(n-1)\geq k(n-1)$ and so 
  \[
  nN\geq k(n-1)+N.
  \]
   Hence, in this case, 
  we see that $x_{[m]}.v=0$ for all $m\geq nN$.   On the other hand, if $n=0$, then we get that 
  $x_{[m]}.v=0$ for all $m\geq N-k$. This is not the sharpest result one has. Indeed, it follows immediately from the vaccum axiom that $x_{[m]}.v=0$ for all $m\geq 0$. The reason we don't obtain this sharp result from our method is that we are not keeping track of the coefficients of the monomials in $S_{i,[n_i]}$. 
 \end{exam}

\ssec{Segal-Sugawara operators} \label{ss:Sugawara}
Let us recall some basic facts about Segal-Sugawara vectors.  Let $\{S_1, \cdots, S_\ell\}\in V_c(\fg)$ be a complete set of Segal-Sugawara vectors.
 Note that the Feigin-Frenkel centre is invariant under the action of the degree operator $t\partial_t$. Therefore, each homogenous component of any Segal-Sugawara vector is again a Segal Sugawara vector. Therefore, without the loss of generality, we may assume that $S_i$ is homogenous of degree $d_i+1$. 
 
\begin{lem} \label{l:SSVector}
The operators $S_i$ can be written as a linear combination of elements of the form $x_{n_1}^{\alpha_1} x_{n_2}^{\alpha_2}\cdots x_{n_k}^{\alpha_k}$ satisfying the following properties:
\begin{itemize} 
\item[(i)] $\sum_{j=1}^k n_j = -(d_i+1)$;
\item[(ii)] $k\leq d_i+1$;
\item[(iii)] $\sum_{j=1}^k \alpha_j =0$.
\end{itemize} 
\end{lem} 

\begin{proof} Part (i) is immediate from the fact that $S_i$ has degree $d_i+1$. Part (ii) follows from (i), because all $n_i$'s are negative. Part (iii) holds since every Segal Sugawara vector is annihilated, in particular, by elements $x\in \fg$. When $x$ runs over the Cartan subalgebra, this means that the weight of each vector is zero. 
\end{proof} 

Let $v$ denote the generating vector of $\bU_{x,r}$. According to \eqref{eq:gxr+}, $v$ is subject to the relation 
\[
x_s^\alpha.v=0, \quad \quad  \forall \,\,s\geq  1-\lceil \alpha(x)-r \rceil.
\]
We will need the following lemma in what follows.

\begin{lem}\label{l:key}
$\displaystyle S_{i,[m]}.v=0,\quad \quad \, \, m\geq - \left(\sum_{j=1}^k \lceil \alpha_j(x)-r \rceil\right)+d_i+1$. 
\end{lem}

\begin{proof} It is easy to check that we always have $1-\lceil \alpha(x)-r \rceil + 1-\lceil \beta(x)-r \rceil\geq 1-\lceil \alpha(x)+\beta(x)-r \rceil$. Therefore, we can use Corollary \ref{l:field} to conclude that $S_{i,[m]}.0=0$ for all 
\[
m\geq \sum_{j=1}^{k} ({r_{\alpha_j}} -n_j)-k =  \sum_{j=1}^k (1-\lceil \alpha(x)-r \rceil -n_j) -k
\]
By the previous lemma, the RHS equals $- \left(\sum_{j=1}^k \lceil \alpha_j(x)-r \rceil\right)+d_i+1$, as required.
\end{proof}

\ssec{Proof of the main theorem}
\label{ss:proof}
Let us first give a reformulation of Theorem \ref{t:support}, which makes clear our strategy for proving it.  Let $V_c(\fg)$ denote the affine vertex algebra at the critical level associated to $\fg$. Let $S_i\in V_c(\fg)$, $i=1,\cdots, \ell$, be a complete set of Segal Sugawara vectors. Let $S_{i,[n_i]}$ denote the corresponding Segal-Sugawara operators. (For a quick introduction to these objects see \cite[\S 2.2]{Molev}.) 
Feigin and Frenkel's Theorem states that the center at the critical level is a completion of the polynomial algebra freely generated on the variables $S_{i,[n_i]}$. It is easy to see that
Theorem \ref{t:support} is equivalent to the following

\begin{thm} \label{t:support'} For $(x,r)\in \cB(G)\times \bR_{\geq 0}$ and all integers $n_i\geq (d_i+1)(r+1)$, the operator 
$S_{i,[n_i]}$ acts trivially on the vacuum vector $v\in \bU_{x,r}$.
\end{thm}

\begin{proof} It is enough to prove the theorem for $x\in \cA$.
Note that $-\lceil \alpha_j(x)-r \rceil \leq -\alpha_j(x)+r$. Since $\sum_{j=1}^k \alpha_j(x)=0$, we conclude that  
$-\sum_{j=1}^k \lceil \alpha_j(x)-r \rceil \leq kr$. By the previous corollary, $S_{i,[m]}$ annihilates $v$ for $m\geq kr+(d_i+1)$. Since $k\leq d_i+1$, we see that $S_{i,[m]}$ annihilates $v$ for all $m\geq (d_i+1)(r+1)$, as required. 
\end{proof}

%%%%%%%%%%%%%%%%%%%%%%%%%%%%%%%%%%%%%%%%%%%%%%%%%

\appendix 
\section{Hecke eigensheaves and the depth conjecture} 
The goal of this appendix is to point out a connection between Conjecture \ref{c:main} and Zhu's conjecture on non-vanishing of Hecke eigensheaves. We note that this section is speculative and not self-contained. We refer the reader to \cite{BD}, \cite{FGross}, and \cite{Zhu}  for the notions we do not define here. 

 Let $X$ be a smooth projective curve over $\bC$. Let $\cG$ be an \emph{integral model} of $G$ over $X$. This means that $\cG$ is a fibre-wise connected smooth affine group scheme over $X$ whose generic fibre is isomorphic to $G$; that is, $\cG_{\bC(X)} = G_{\bC(X)}$. We let $\Bun_\cG(X)$ denote the moduli stack of $\cG$-torsors on $X$. One can define an analogue of Hitchin's map whose source is the cotangent stack $\T^*\Bun_\cG(X)$, cf. \cite{Zhu}.  In \emph{op. cit.}, Zhu indicates how the Beilinson-Drinfeld quantisation machinery \cite{BD} can be applied to this twisted setting. The input of this machine is a (global) oper $\chi\in \Loc_\hG(X)$, where $\chi$ is allowed to have meromorphic singularity at finitely many places. The output is a Hecke eigensheaf $\Aut_\chi$ on $\Bun_\cG$ with eigenvalue the local system underlying $\chi$.  Zhu conjectures that these Hecke eigensheaves are non-zero. 

Let us see how these global considerations help for constructing modules satisfying (i) and (ii). For $x\in X$,  let $\cO_x$ denote the completed local ring at $X$ and let $F_x$ denote its fraction field. Let $K_x=\cG(\cO_x)$.  One has a canonical morphism $G(F_x)/K_x\ra \Bun_\cG$. 
Let $\omega$ be the canonical line bundle on $\Bun_\cG$ and let $\cL_{c}$ denote the pull back of $\omega^{1/2}$ to $G(F_x)/K_x$. Let $\delta_e$ denote the delta $D$-module on $\Gr_{\cG,x}$ twisted by $\cL_{c}$. Let 
\[
\Vac_x:=\Gamma(\Gr_{\cG,x}, \delta_e). 
\]
In most (but not all) situations, we have $\Vac_x\simeq \Ind_{\Lie(K_x)\oplus \bC}^{\hfg}(\bC)$. 
By construction, we have $\Vac_x\in \hfg_{c,x}-\modd$. Let $\chi(x)$ denote the restriction of $\chi$ to $x$. Let $\Vac_x(\chi)$ denote the central reduction of $\Vac_x$ at $\chi(x)$. Using an easy global to local argument, we obtain:

\begin{lem} If the Hecke eigensheaf $\Aut_\chi$ is non-zero, then so is the module $\Vac_x(\chi)$. 
\end{lem} 

The module $\bV=\Vac_x(\chi)$ is, therefore, usually a $K_x$-integrable module with central support $\chi$. 

As an example of the success of this strategy, we can show that if $\chi$ is an oper of slope $1/h$, where $h$ is the Coxeter number of $\hG$, and if $\chi$ can be extended to a global oper on $\mathbb{P}^1$ as in \cite{FGross} and \cite{Zhu}, then the central reduction $\bW(\chi)$ is non-zero.\footnote{This fact is stated, without proof, in the last paragraph of \cite{FGross}.} Here, $\bW=\bW(\Psi)$ denotes the affine Whittaker module, cf. \cite{FGross}. It is easy to see that this module is $G_{x,1/h^+}$-integrable, where $x$ is the barycentre of the interior of an alcove. On the other hand, according to \cite{Zhu}, the global automorphic sheaves produced in this case are non-zero. The above lemma then implies that $\bW(\chi)$ is a non-zero $G_{x,1/h^+}$-integrable object in $\hfg_c-\modd_{\chi}$. Hence, for such opers $\chi$, we have 
\begin{equation} 
 d(\hfg_c-\modd_\chi) \leq 1/h. 
 \end{equation} 

Returning to the general situation, we are left with the following question: is it always possible to extend a local oper $\chi\in \Op_\hG(\cDt)$ to a global oper on $X$? We note that it is always possible to extend a meromorphic connection on the punctured disk to a global connection on $\mathbb{P}^1$ with regular singularity at zero, cf. \cite{Katz86}. The corresponding question for opers seems to be open.

 %%%%% %%%%% %%%%% %%%%% %%%%% %%%%% %%%%% %%%%% %%%%%

 %\section{References} 

 \end{document}